\documentclass[12pt]{article}

\usepackage{amssymb}
\usepackage{amsmath}
\usepackage{amsthm}
\usepackage{amsfonts}
\usepackage{bbm}

\newcommand{\eee}{{\rm e}}

\DeclareMathOperator{\1}{\mathbbm{1}}

\newcommand{\mmp}{\mathbb{P}}

\newcommand{\dod}{\overset{{\rm d}}{\to}}

\newcommand{\me}{\mathbb{E}}
\newcommand{\mr}{\mathbb{R}}
\newcommand{\mn}{\mathbb{N}}

\newtheorem{thm}{Theorem}[section]
\newtheorem{lemma}[thm]{Lemma}

\newtheorem{cor}[thm]{Corollary}

\newtheorem{assertion}[thm]{Proposition}

\theoremstyle{definition}

\theoremstyle{remark}

\begin{document}
\title{A law of the iterated logarithm for the number of blocks in regenerative compositions generated by gamma-like subordinators}
\date{\today}

\author{Alexander Iksanov\footnote{ Faculty of Computer Science and Cybernetics, Taras Shevchenko National University of Kyiv, 01033 Kyiv, Ukraine, e-mail:
iksan@univ.kiev.ua} ~~~and~~ Wissem Jedidi \footnote{Department of
Statistics  \& OR, King Saud University, P.O. Box 2455, Riyadh
11451, Saudi Arabia and Universit\'e de Tunis El Manar, Facult\'e
des Sciences de Tunis, LR11ES11 Laboratoire d'Analyse
Math\'ematiques et Applications, 2092, Tunis, Tunisia, e-mail:
wissem$_-$jedidi@yahoo.fr} }

\maketitle

\begin{abstract}
\noindent
The points of the closed range of a drift-free subordinator with no killing are used for separating into blocks the elements of a sample of size $n$ from the standard exponential distribution. This gives rise to a random composition of $n$. Assuming that the subordinator has the L\'{e}vy measure, which behaves near zero like the gamma subordinator, we prove a law of the iterated logarithm for the number of blocks in the composition as $n$ tends to infinity. Along the way we prove a law of the iterated logarithm for the Lebesgue convolution of a standard Brownian motion and a deterministic regularly varying function. This result may be of independent interest.
\end{abstract}


\section{Introduction and main result}\label{sect:intro}

Let $S:=(S(t))_{t\geq 0}$ be a subordinator (an increasing L{\'e}vy process) with $S(0)=0$, zero drift, no killing and a nonzero L{\'e}vy measure $\nu$. % defined by
%\begin{equation*} %\label{eq:measure}
%\nu({\rm d}x)=\frac{\eee^{-x}}{1-\eee^{-x}}\1_{(0,\infty)}(x)\,{\rm d}x.
%\end{equation*}
Let $E_1,E_2,\ldots$ % $E_1,\ldots, E_n$
be independent random variables with the exponential distribution of unit mean, which are independent of $S$. The closed range of $S$ has zero Lebesgue measure and splits the positive halfline into infinitely many disjoint intervals that we call gaps. Assuming that there are $n$ variables $E_1,\ldots, E_n$, we call a gap occupied if it contains at least one $E_j$, $j=1,2,\ldots, n$. The sequence of positive occupancy numbers of the gaps, written in the left-to-right order,
is a composition $\mathcal{C}_n$ of integer $n$. The sequence $(\mathcal{C}_n)_{n\geq 1}$ forms a regenerative composition structure as introduced and discussed in \cite{Gnedin+Pitman:2005}. Denote by $K_n$ the number of occupied gaps. We call $K_n$ the number of blocks of the regenerative composition of $n$. Observe that the variables $K_1$, $K_2,\ldots$ being functions of $S$ and $E_1,E_2,\ldots$ live on a common probability space. Thus, investigating their almost sure (a.s.) asymptotic behavior, particularly proving a law of the iterated logarithm (LIL), makes sense.

Assume that $S$ is a compound Poisson process. This is equivalent to finiteness of the L\'{e}vy measure $\nu$, that is, $\nu((0,\infty))<\infty$. Under this assumption, the following LIL was proved in our earlier paper \cite{Iksanov+Jedidi+Bouzzefour:2017}. For a family $(x_t)$ of real numbers denote by $C((x_t))$ the set of its limit points.
\begin{assertion}\label{main}
Let $\xi$ be a random variable having the same distribution as jumps of a compound Poisson process $S$. Assume that ${\tt s}^2:={\rm Var}\,[\xi]\in (0,\infty)$ and that, for some $a>0$, $\me
[|\log(1-\eee^{-\xi})|^a]<\infty$. Then
$$C\bigg(\bigg(\frac{K_{\lfloor \eee^n\rfloor}-{\tt m}^{-1}\int_0^n\mmp\{|\log(1-\eee^{-\xi})|\leq x\}{\rm d}x}{(2{\tt s}^2{\tt m}^{-3}n\log\log
n)^{1/2}}:n\geq 3\bigg)\bigg)=[-1,1]\quad\text{{\rm a.s.}}$$ In
particular,
$${\lim\sup\,(\lim
\inf)}_{n\to\infty}\frac{K_{\lfloor \eee^n\rfloor}-{\tt m}^{-1}\int_0^n\mmp\{|\log(1-\eee^{-\xi})|\leq
x\}{\rm d}x}{(2n\log\log n)^{1/2}}=+(-){\tt s}{\tt m}^{-3/2}\quad\text{{\rm a.s.}}$$
\end{assertion}

Put $$\Phi(t):=\int_{(0,\infty)} (1-\exp\{-t(1-\eee^{-x})\})\nu({\rm  d}x),\quad t>0$$ and assume that
\begin{equation}\label{eq:regular}
\varphi(t):=\Phi(\eee^t)~\sim~ t^\beta \ell(t),\quad t \to \infty
\end{equation}
for some $\beta>0$ and some $\ell$ slowly varying at $\infty$. In particular, this implies that the L\'{e}vy measure $\nu$ is infinite, that is, we are beyond the compound Poisson case.

To date, there are several methods for proving a central limit theorem (CLT) for $K_n$ under \eqref{eq:regular}, see \cite{Barbour+Gnedin:2006, Gnedin+Iksanov:2012, GnePitYor1}. Now we cite a fragment of Theorem 3.1(a) in \cite{Gnedin+Iksanov:2012}. Introduce the notation $$\sigma^2 :={\rm Var}[S(1)]=\int_{(0,\infty)}x^2 \nu({\rm d}x),\quad \mu:=\me [S(1)]=\int_{(0,\infty)} x\nu({\rm d}x).$$
\begin{assertion}\label{prop:clt}
Suppose \eqref{eq:regular} and $\sigma^2\in (0,\infty)$. Then $$\frac{K_n-\mu^{-1}\int_1^n x^{-1}\Phi(x){\rm d}x}{(\sigma^2\mu^{-3}\log n)^{1/2}\Phi(n)}~\dod~\beta\int_0^1 B(1-x)x^{\beta-1}{\rm d}x,$$ where $(B(u))_{u\geq 0}$ is a standard Brownian motion.
\end{assertion}

Observe that the limit random variable has the normal distribution with mean $0$ and variance $(2\beta+1)^{-1}$. Typically, a CLT provides a hint concerning a possible form of a LIL: one should use the same centering, whereas the normalization has to be multiplied by square root of two times the iterated logarithm. The argument of the iterated logarithm should be chosen in the natural scale of the CLT, which is $\log n$ in Proposition \ref{prop:clt}. We intend to show that this idea works smoothly in the present context. However, there is a minor complication. We cannot prove a LIL under the sole assumption \eqref{eq:regular}. We need an additional property that $\varphi^\prime$ is regularly varying at $\infty$ of index $\beta-1$. To ensure this as well as \eqref{eq:regular} we assume that, for all $\lambda>0$,
\begin{equation}\label{eq:Haan}
\lim_{t\to\infty}\frac{\Phi(\lambda t)-\Phi(t)}{\beta (\log t)^{\beta-1}\ell(\log t)}=\log \lambda.
\end{equation}
It will be explained in Lemma \ref{lem:haan} that condition \eqref{eq:Haan} ensures \eqref{eq:regular} and is equivalent to
\begin{equation}\label{eq:derivative}
\varphi^\prime(t)~\sim~ \beta t^{\beta-1}\ell(t),\quad t\to\infty.
\end{equation}

Given next is our main result.
\begin{thm}\label{thm:main}
Suppose \eqref{eq:Haan} and $\sigma^2\in (0,\infty)$. Then
\begin{equation}\label{eq:inter10}
C\bigg(\bigg(\frac{K_n-\mu^{-1}\int_1^n x^{-1}\Phi(x){\rm d}x}{(2\sigma^2\mu^{-3}(2\beta+1)^{-1}\log n \log\log\log n)^{1/2}\Phi(n)}: n~\text{{\rm large enough}}\bigg)\bigg)=[-1,1]\quad\text{{\rm a.s.}}
\end{equation}
\end{thm}

The proof of Proposition \ref{main} given in \cite{Iksanov+Jedidi+Bouzzefour:2017} was based on an a.s.\ approximation of $K_{\lfloor \eee^n\rfloor}$ by $\sum_{k\geq 1}\1_{\{\xi_1+\ldots+\xi_{k-1}-\log(1-\eee^{-\xi_k})\leq n\}}$, where $\xi_1$, $\xi_2,\ldots$ are independent copies of a random variable $\xi$ having the same distribution as jumps of a compound Poisson process $S$. There is a natural enumeration of gaps, and the random hitting probability of the $k$th gap is $\eee^{-\xi_1-\ldots-\xi_{k-1}+\log(1-\eee^{-\xi_k})}$, so that the approximating quantity can be thought of as the number of gaps with a `large' hitting probability. Roughly speaking, the argument in \cite{Iksanov+Jedidi+Bouzzefour:2017} essentially exploited the fact that a compound Poisson process admits a.s.\ finitely many jumps within each finite time-interval. In the case where $S$ is not a compound Poisson process (which is the setting of Theorem \ref{thm:main}) the range of $S$ has topology of a Cantor set, and the aforementioned gap-counting is no longer adequate. Our proof of Theorem \ref{thm:main} uses a two-stage a.s.\ approximation of the Poissonized number of blocks obtained by replacing $n$ with $\pi(t)$, where $(\pi(t))_{t\geq 0}$ is a Poisson process which is independent of everything else. The Poissonized number of blocks is first approximated by its conditional mean which in its turn is approximated by the Lebesgue convolution of a standard Brownian motion and a deterministic function. Of course, at the end we have to de-Poissonize, that is, to get back to the original version with the exponential sample of size $n$.

A list of typical subordinators $S$ satisfying \eqref{eq:Haan} includes the gamma subordinators with the L\'{e}vy measures $\nu$ given by
\begin{equation}\label{eq:gamma}
\nu({\rm d}x)=\frac{\theta\eee^{-\lambda x}}{x}\1_{(0,\infty)}(x){\rm d}x
\end{equation}
for some positive $\theta$ and $\lambda$ and closely related subordinators with the L\'{e}vy measures given by
\begin{equation}\label{eq:gammalike}
\nu({\rm d}x)=\frac{\theta\eee^{-\lambda x}}{1-\eee^{-x}}\1_{(0,\infty)}(x){\rm d}x.
\end{equation}
This observation justifies the title of the paper.

We close this section by giving specializations of Theorem \ref{thm:main} to subordinators with the L\'{e}vy measures \eqref{eq:gamma} and \eqref{eq:gammalike}.
\begin{cor}\label{cor:appl}
(a) Suppose \eqref{eq:gamma}. Then
\begin{equation*}\label{eq:inter10}
C\bigg(\bigg(\frac{K_n-(\lambda/2)(\log n)^2}{(2\lambda (\log n)^3 \log\log\log n)^{1/2}}: n~\text{{\rm large enough}}\bigg)\bigg)=[-1,1]\quad\text{{\rm a.s.}}
\end{equation*}

\noindent (b) Suppose \eqref{eq:gammalike}. Then
\begin{equation*}\label{eq:inter11}
C\bigg(\bigg(\frac{K_n-(2\Psi^\prime(\lambda))^{-1}(\log n)^2}{(2|\Psi^{\prime\prime}(\lambda)|(\Psi^\prime(\lambda))^{-3}(\log n)^3 \log\log\log n)^{1/2}}: n~\text{{\rm large enough}}\bigg)\bigg)=[-1,1]\quad\text{{\rm a.s.}},
\end{equation*}
where $\Psi$ is the logarithmic derivative of the gamma function.
\end{cor}

The fact that $\theta$ does not appear in Corollary \ref{cor:appl} is not surprising. Multiplying a L\'{e}vy measure by a positive factor leads to a time-change of the corresponding subordinator, which does not affect
$K_n$.

The remainder of the paper is structured as follows. After giving a number of auxiliary results in Section \ref{sect:aux}, we prove Theorem \ref{thm:main} in Section \ref{sect:main} and Corollary \ref{cor:appl} in Section \ref{sect:appl}.

\section{Auxiliary results}\label{sect:aux}

We first justify the claim made in Section \ref{sect:intro}.
\begin{lemma}\label{lem:haan}
Condition \eqref{eq:Haan} ensures \eqref{eq:regular} and is equivalent to \eqref{eq:derivative}.
\end{lemma}
\begin{proof}
Condition \eqref{eq:Haan} ensures \eqref{eq:regular} according to the implication (3.7.6) $\Rightarrow$ (3.7.8) of Theorem 3.7.3 in \cite{BGT}. Since $\Phi^\prime$ is a nonincreasing function, condition \eqref{eq:Haan} is equivalent to $$\Phi^\prime(t)\sim \frac{\beta (\log t)^{\beta-1}\ell(\log t)}{t},\quad t\to\infty$$ by Theorem 3.6.8 in \cite{BGT}. The latter relation is equivalent to \eqref{eq:derivative}.
\end{proof}

We proceed with a strong approximation result. For a subordinator $S$, put $S^\leftarrow(t):=\inf\{u\geq 0: S(u)>t\}$ for $t\geq 0$, so that $S^\leftarrow$ is the generalized inverse (random) function of $S$.
%For the proof of Theorem~\ref{main} we shall need the following strong approximation result, which follows, for instance, from Theorem 12.13 on p.~227 in \cite{Kallenberg:1997}.
\begin{lemma}\label{chs}
Assume that $\sigma^2={\rm Var}[S(1)]=\int_{(0,\infty)}x^2\nu({\rm d}x)\in (0,\infty)$. Then there exists
a standard Brownian motion $W$ such that
\begin{equation}\label{eq:appr}
\lim_{t\to\infty}\frac{\sup_{0\leq u\leq
t}\,\big|S^\leftarrow(u)-\mu^{-1}u-\sigma\mu^{-3/2}W(u)\big|}{(t\log\log t)^{1/2}}=0\quad\text{{\rm
a.s.}},
\end{equation}
where $\mu=\me [S(1)]=\int_{(0,\infty)} x\nu({\rm d}x)<\infty$. In particular,
\begin{equation}\label{eq:lil}
{\lim\sup}_{n\to\infty}\frac{\sup_{y\in [0,\,n]}|S^\leftarrow(y)-\mu^{-1}y|}{(2n\log\log n)^{1/2}}=\frac{\sigma}{\mu^3}\quad\text{{\rm a.s.}}
\end{equation}
\end{lemma}
\begin{proof}
By Theorem 12.13 on p.~227 in \cite{Kallenberg:1997}, relation \eqref{eq:appr} holds with $\tau(u):=\inf\{v\geq 0: S(\lfloor v\rfloor)>u\}$ replacing $S^\leftarrow(u)$. Noting that $\tau(u)-S^\leftarrow(u)\in [0,1]$ a.s.\ completes the proof of \eqref{eq:appr}.

The LIL \eqref{eq:lil} follows from \eqref{eq:appr} and the corresponding LIL for a Brownian motion. Also, \eqref{eq:lil} is a consequence of the LIL for $\tau$ given in Proposition 3.5 of \cite{Iksanov+Jedidi+Bouzzefour:2017}.
\end{proof}

The next lemma collects several properties of $S^\leftarrow$ to be used in the paper.

\begin{lemma}\label{lem:mom}
Under the sole assumption that $S$ is nondegenerate at $0$,

\noindent (a) for all $r>0$ and all $h>0$, $\me [(S^\leftarrow(h))^r]<\infty$;

\noindent (b) for all $r>0$, $\me [(S^\leftarrow(t))^r]]=O(t^r)$ as $t\to\infty$;

\noindent (c) for all $h>0$ and all $\delta>0$, $$\lim_{n\to\infty}\frac{S^\leftarrow(n+h)-S^\leftarrow(n)}{n^\delta}=0\quad\text{{\rm a.s.}}$$
\end{lemma}
\begin{proof}
Part (a) is justified by
\begin{multline*}
r^{-1}\me [(S^\leftarrow(h))^r]=\int_0^\infty x^{r-1}\mmp\{S^\leftarrow(h)>x\}{\rm d}x=\int_0^\infty x^{r-1}\mmp\{S(x)\leq h\}{\rm d}x\\\leq \eee^h\int_0^\infty x^{r-1}\me \eee^{-S(x)}{\rm d}x=\eee^h\int_0^\infty x^{r-1}\eee^{-\kappa x}{\rm d}x<\infty.
\end{multline*}
Here, $\kappa:=\int_{(0,\,\infty)} (1-\eee^{-x})\nu({\rm d}x)<\infty$, and we have used Markov's inequality.

Noting that $S^\leftarrow(t)\leq \tau(t)$ a.s., part (b) follows from Theorem 5.1 on p.~57 in \cite{Gut:2009}.

As for part (c), fix any $\delta>0$ and pick $r>0$ such that $r\delta>1$. By the strong Markov property of $S$ and part (a) of the lemma, $\me [(S^\leftarrow(n+h)-S^\leftarrow(n))^r]\leq \me [(S^\leftarrow(h))^r]<\infty$. %$\me [(J_n)^2]\leq (\varphi(1))^2 \me [(S^\leftarrow(1))^2]<\infty$.
Appealing now to Markov's inequality and the direct part of the Borel-Cantelli lemma we arrive at the claim.
\end{proof}
Now we state a LIL % law of the iterated logarithm
for the Lebesgue convolution of a standard Brownian motion and a regularly varying function. This is a generalization of Theorem 1 in \cite{Lachal:1997} which treats the case where $f(x)=x^n$ for $x\geq 0$ and some $n\in\mn$. Proposition \ref{prop:lil} % the result
may be of independent interest.
\begin{assertion}\label{prop:lil}
Let $\alpha>0$, $f$ be a nonnegative differentiable function with its derivative $f^\prime$ being regularly varying at $\infty$ of index $\alpha-1$, and $B$ a standard Brownian motion. Then
$$C\bigg(\bigg(\frac{\int_0^t B(t-x)f^\prime(x){\rm d}x}{(2(2\alpha+1)^{-1}t\log\log t)^{1/2}f(t)}: t>\max(\eee, t_0)\bigg)\bigg)=[-1,1]\quad\text{{\rm a.s.}},$$ where $t_0>0$ is any fixed value such that $f(t)>0$ for all $t\geq t_0$.
\end{assertion}
\begin{proof}
We shall use a decomposition
\begin{multline*}
\frac{\int_0^t B(t-x)f^\prime(x){\rm d}x}{f(t)}=\frac{\alpha}{t^\alpha}\int_0^t B(t-x)x^{\alpha-1}{\rm d}x+\int_0^t B(t-x)\Big(\frac{f^\prime(x)}{f(t)}-\frac{\alpha x^{\alpha-1}}{t^\alpha}\Big){\rm d}x. %\\=: R_1(t)+R_2(t).
%\label{eq:second}
\end{multline*}
Denote the second integral on the right-hand side by $R_2(t)$. We first show that
\begin{equation}\label{eq:inter1}
|R_2(t)|=o((t\log\log t)^{1/2}),\quad t\to\infty\quad\text{a.s.}
\end{equation}
To this end, write
\begin{equation}\label{eq:second}
|R_2(t)|\leq \int_0^t |B(t-x)|\Big|\frac{f^\prime(x)}{f(t)}-\frac{\alpha x^{\alpha-1}}{t^\alpha}\Big|{\rm d}x\leq \sup_{u\in [0,\,t]}|B(u)| \int_0^1 \Big|\frac{tf^\prime(ty)}{f(t)}-\alpha y^{\alpha-1}\Big| {\rm d}y.
\end{equation}
By the LIL for Brownian motion, $$\sup_{u\in [0,\,t]}|B(u)|=O((t\log\log t)^{1/2}),\quad t\to\infty\quad\text{a.s.}$$
Thus, it is enough to show that the integral on the right-hand side of \eqref{eq:second} vanishes as $t\to\infty$. Fix any $\varepsilon\in (0,1)$. Since $f^\prime$ is regularly varying of index $\alpha-1$, we infer $$\lim_{t\to\infty}\sup_{u\in[\varepsilon,\,1]}|f^\prime(tu)/f^\prime(t)-u^{\alpha-1}|=0$$ by the uniform convergence theorem for regularly varying functions (Theorem 1.5.2 in \cite{BGT}) and also $\lim_{t\to\infty}(tf^\prime(t)/f(t))=\alpha$. Regular variation of $f^\prime$ of index $\alpha-1$ entails regular variation of $f$ of index $\alpha$. Hence, $\lim_{t\to\infty}\sup_{u\in [\varepsilon,\,1]}\big|tf^\prime(tu)/f(t)-\alpha u^{\alpha-1}\big|=0$ and thereupon $$\lim_{t\to\infty}\int_\varepsilon^1 \Big|\frac{tf^\prime(ty)}{f(t)}-\alpha y^{\alpha-1}\Big| {\rm d}y=0.$$ Finally, $$\int_0^\varepsilon \Big|\frac{tf^\prime(ty)}{f(t)}-\alpha y^{\alpha-1}\Big|{\rm d}y\leq \int_0^\varepsilon \Big(\frac{tf^\prime(ty)}{f(t)}+\alpha y^{\alpha-1}\Big){\rm d}y\leq \frac{f(t\varepsilon)}{f(t)}+\varepsilon^\alpha~\to~2\varepsilon^\alpha,\quad t\to\infty.$$ Sending $\varepsilon \to 0+$ completes the proof of \eqref{eq:inter1}.

In view of $\alpha\int_0^t B(t-x)x^{\alpha-1}{\rm d}x=\int_0^t (t-x)^\alpha{\rm d}B(x)$ for $t>0$, it remains to prove that
\begin{equation}\label{eq:inter2}
C\bigg(\bigg(\frac{\int_0^t (t-x)^\alpha{\rm d}B(x)}{(2(2\alpha+1)^{-1}t^{2\alpha+1}\log\log t)^{1/2}}: t>\eee\bigg)\bigg)=[-1,1]\quad\text{{\rm a.s.}}
\end{equation}
The latter integral is understood as a Skorokhod integral. It is well-defined as a consequence of $\int_0^t (t-x)^{2\alpha}{\rm d}x<\infty$. Assume that we can prove that
\begin{equation}\label{eq:inter3}
{\lim\sup}_{t\to\infty}\frac{\int_0^t (t-x)^\alpha{\rm d}B(x)}{(t^{2\alpha+1}\log\log t)^{1/2}}=\Big(\frac{2}{2\alpha+1}\Big)^{1/2}\quad\text{a.s.}
\end{equation}
Since $-B$ is also a Brownian motion, the latter entails $${\lim\inf}_{t\to\infty}\frac{\int_0^t (t-x)^\alpha{\rm d}B(x)}{(t^{2\alpha+1}\log\log t)^{1/2}}=-\Big(\frac{2}{2\alpha+1}\Big)^{1/2}\quad\text{a.s.}$$ and thereupon \eqref{eq:inter2} because the random function $t\mapsto (t^{2\alpha+1}\log\log t)^{-1/2} \int_0^t (t-x)^\alpha{\rm d}B(x)$ is a.s.\ continuous on $(\eee,\infty)$.

\noindent {\sc Proof of \eqref{eq:inter3}.} Assume that $\alpha\in\mn$. In this case relation \eqref{eq:inter3} is proved in Theorem 1 of \cite{Lachal:1997}. Thus, assume in what follows that $\alpha\notin \mn$. The proof of the fact that the upper limit does not exceed the right-hand side of \eqref{eq:inter3} mimics\footnote{In the cited article the proof is only given in the situation that $t\to 0+$. The case where $t\to\infty$ requires an obvious modification.} the proof of Theorem 1 in \cite{Lachal:1997}.

Let ${\rm Normal}\,(0,1)$ denote a random variable with the standard normal distribution. Left with showing that the upper limit is not smaller than the right-hand side of \eqref{eq:inter3}, we pick a $\theta>1$ and note that, for each $j\in\mn$, the random variable $\int_{\theta^{j-1}}^{\theta^j}(\theta^j-x)^\alpha{\rm d}B(x)$ has the same distribution as $$\Big(\int_{1/\theta}^1 (1-x)^{2\alpha}{\rm d}x\Big)^{1/2}\theta^{j(\alpha+1/2)}{\rm Normal}\,(0,1)=\Big(\frac{(1-1/\theta)^{2\alpha+1}}{2\alpha+1}\Big)^{1/2}\theta^{j(\alpha+1/2)}{\rm Normal}\,(0,1).$$ Furthermore, for different positive integer $j$ these random variables are independent. Thus, if we can prove that
\begin{equation}\label{eq:inter5}
\sum_{j\geq 1}\mmp(T_j)=\infty,
\end{equation}
where $$T_j:=\Big\{\int_{\theta^{j-1}}^{\theta^j}(\theta^j-x)^\alpha{\rm d}B(x)\geq \Big(\frac{2(1-1/\theta)^{2\alpha+1}}{2\alpha+1}\Big)^{1/2}\theta^{j(\alpha+1/2)}(\log\log \theta^j)^{1/2}\Big\},\quad j\in\mn,$$ then, by the converse part of the Borel-Cantelli lemma, $\mmp\{T_j~\text{i.o.}\}=1$. Using $$\int_x^\infty \eee^{-y^2/2}{\rm d}y~\sim~\frac{\eee^{-x^2/2}}{x},\quad x\to\infty,$$ we infer $$\mmp(T_j)=\mmp\{{\rm Normal}\,(0,1)\geq (2\log\log \theta^j)^{1/2}\}~\sim~\frac{1}{2\pi^{1/2}\log\theta}\frac{1}{j(\log j)^{1/2}},\quad j\to\infty.$$ This proves \eqref{eq:inter5}.

Next, we show that
\begin{equation}\label{eq:inter6}
{\lim\sup}_{j\to\infty}\frac{\int_0^{\theta^{j-1}}(\theta^j-x)^\alpha{\rm d}B(x)}{(\theta^{j(2\alpha+1)}\log\log \theta^j)^{1/2}}\leq 2\Big(\frac{2(1-(1-1/\theta)^{2\alpha+1})}{2\alpha+1}\Big)^{1/2}\quad\text{a.s.}
\end{equation}
For each $j\in\mn$, the random variable $\int_0^{\theta^{j-1}}(\theta^j-x)^\alpha{\rm d}B(x)$ has the same distribution as $$\Big(\frac{(1-1/\theta)^{2\alpha+1}}{2\alpha+1}\Big)^{1/2}\theta^{j(\alpha+1/2)}{\rm Normal}\,(0,1).$$ Put $$Q_j:=\Big\{\int_0^{\theta^{j-1}}(\theta^j-x)^\alpha{\rm d}B(x)>2\Big(\frac{2(1-(1-1/\theta)^{2\alpha+1})}{2\alpha+1}\Big)^{1/2}\theta^{j(\alpha+1/2)}(\log\log \theta^j)^{1/2} \Big\}.$$ Since $$\mmp(Q_j)=\mmp\{{\rm Normal}\,(0,1)>2(2\log\log \theta^j)^{1/2}\}~\sim~\frac{1}{4\pi^{1/2}(\log \theta)^4}\frac{1}{j^4 (\log j)^{1/2}},\quad j\to\infty,$$ we conclude that $\sum_{j\geq 1}\mmp(Q_j)<\infty$. Now \eqref{eq:inter6} follows by the direct part of the Borel-Cantelli lemma.

Since $-B$ has the same distribution as $B$, relation \eqref{eq:inter6} entails that a.s., for $j$ large enough, $$\int_0^{\theta^{j-1}}(\theta^j-x)^\alpha{\rm d}B(x)\geq -3\Big(\frac{2(1-(1-1/\theta)^{2\alpha+1})}{2\alpha+1}\Big)^{1/2} (\theta^{j(2\alpha+1)}\log\log \theta^j)^{1/2}.$$ This in combination with $\mmp\{T_j~\text{i.o.}\}=1$ ensures that a.s. $$\frac{\int_0^{\theta^j}(\theta^j-x)^\alpha{\rm d}B(x)}{(\theta^{j(2\alpha+1)}\log\log \theta^j)^{1/2}}\geq \Big(\frac{2}{2\alpha+1}\Big)^{1/2} \big((1-1/\theta)^{\alpha+1/2}-3 (1-(1-1/\theta)^{2\alpha+1})^{1/2}\big)$$ for infinitely many $j$. Hence, for each fixed $\theta>1$, $${\lim\sup}_{t\to\infty}\frac{\int_0^t(t-x)^\alpha{\rm d}B(x)}{(t^{2\alpha+1}\log\log t)^{1/2}}\geq \Big(\frac{2}{2\alpha+1}\Big)^{1/2} \big((1-1/\theta)^{\alpha+1/2}-3 (1-(1-1/\theta)^{2\alpha+1})^{1/2}\big) \quad\text{a.s.}$$ Since the second factor on the right-hand side approaches $1$ from below as $\theta\to\infty$, we arrive at $${\lim\sup}_{t\to\infty}\frac{\int_0^t(t-x)^\alpha{\rm d}B(x)}{(t^{2\alpha+1}\log\log t)^{1/2}}\geq \Big(\frac{2}{2\alpha+1}\Big)^{1/2} % \big((1-1/\theta)^{\alpha+1/2}-3 (1-(1-1/\theta)^{2\alpha+1})^{1/2}\big)
 \quad\text{a.s.}$$ thereby completing the proof.
\end{proof}

\section{Proof of Theorem \ref{thm:main}}\label{sect:main}

Denote by $(S_k)_{k\geq 1}$ a standard random walk with independent increments having the exponential
distribution of unit mean. Put $\pi(t):= \#\{k\in\mn: S_k\leq t\}$ for $t\geq 0$, so that $\pi:= (\pi(t))_{t\geq 0}$ is a Poisson process on $[0,+\infty)$ of unit intensity. It is assumed that $\pi$ is independent of both $S$ and the infinite sample $E_1$, $E_2,\ldots$

For each $t\geq 0$, denote by $\mathcal{K}(t)$ the number of gaps that contain at least one element of the Poissonized sample $E_1,\ldots, E_{\pi(t)}$. Recalling the notation $\varphi(t)=\Phi(\eee^t)$ put, for $t\in\mr$,
\begin{multline}\label{co}
A(t):=\int_0^\infty\varphi(t-S(v)){\rm d}v=\int_{[0,\,\infty)}\varphi(t-x){\rm d}S^\leftarrow(x)\\=\int_{[0,\,t]}\varphi(t-x){\rm d}S^\leftarrow(x)%\int_{[0,\,t]}S^\leftarrow(t-x)\varphi^\prime(x){\rm d}x
+\int_{(t,\,\infty)}\varphi(t-x){\rm d}S^\leftarrow(x)=:A_1(t)+A_2(t).
\end{multline}

The subsequent proof will be divided into three steps. The purpose of the first step is to show that the large time a.s.\ asymptotic behavior of $\mathcal{K}(\eee^t)$ is driven by that of $A_1(t)$.

\noindent {\sc Step 1}. We intend to prove that
\begin{equation}\label{eq:appr1}
\lim_{t\to\infty}\frac{\mathcal{K}(\eee^t)-A_1(t)}{t^{1/2}\varphi(t)}=0\quad\text{a.s.}
\end{equation}

By Lemma 6.4 in \cite{Gnedin+Pitman+Yor:2006}, $\mathcal{K}(\eee^t)-A(t)$ is the terminal value of a square integrable martingale $(M_t(u))_{u\in [0,\infty]}$ say, with right-continuous paths and unit jumps. Furthermore, the variable $A(t)$ is the terminal value of the quadratic predictable characteristics of $(M_t(u))_{u\in [0,\infty]}$. For each $u>0$ and each $r>0$, $\sum_{0<s\leq u}(M_t(u)-M_t(u-))^{2r}=M_t(u)$. Hence, the terminal value of a compensator for $\Big(\sum_{0<s\leq u}(M_t(u)-M_t(u-))^{2r}\Big)_{u\in[0,\infty]}$ is $A(t)$. With this at hand, an application of Corollary 2.4 in \cite{Hernandez+Jacka:2022} yields, for each $r\geq 1$ and appropriate positive constant $C_r$, $$\me [|\mathcal{K}(\eee^t)-A(t)|^{2r}]\leq C_r(\me [(A(t))^r]+\me [A(t)]).$$ Now we have to find out what is the order of growth of the right-hand side. Using monotonicity of $\varphi$ we obtain $(A_1(t))^r=\Big(\int_{[0,\,t]}\varphi(t-x){\rm d}S^\leftarrow(x)\Big)^r\leq (\varphi(t))^r (S^\leftarrow(t))^r$. Hence, by Lemma \ref{lem:mom}, $\me [(A_1(t))^r]=O((t\varphi(t))^r$ as $t\to\infty$.

The inequality $\int_{(0,\infty)}\min(x,1)\nu({\rm d}x)<\infty$ which holds true for any L\'{e}vy measure $\nu$ is equivalent to $\Phi^\prime(0)<\infty$. This ensures that the function $\varphi$ is Lebesgue integrable on $(-\infty,0]$. In particular, $\sum_{n\leq 0}\varphi(n)<\infty$. Further, invoking monotonicity of $\varphi$ yields $$\varphi(t)\leq \sum_{n\leq 0}\varphi(n)\1_{[n-1,\, n)}(t),\quad t<0,$$ whence
\begin{multline}
(A_2(t))^r\leq \Big(\int_{(t,\infty)}\sum_{n\leq 0}\varphi(n)\1_{[n-1,\, n)}(t-x){\rm d}S^\leftarrow(x)\Big)^r\\\leq \Big( \sum_{n\geq 0}\varphi(n)(S^\leftarrow(t-n+1)-S^\leftarrow(t-n))\Big)^r\\=\Big(\sum_{j\leq 0}\varphi(j)\Big)^r\Big( \sum_{n\geq 0}\frac{\varphi(n)}{\sum_{j\leq 0}\varphi(j)}(S^\leftarrow(t-n+1)-S^\leftarrow(t-n))\Big)^r\\\leq \Big(\sum_{j\leq 0}\varphi(j)\Big)^r \sum_{n\geq 0}\frac{\varphi(n)}{\sum_{j\leq 0}\varphi(j)}(S^\leftarrow(t-n+1)-S^\leftarrow(t-n))^r. \label{eq:inter78}
\end{multline}
Here, the last inequality is justified by convexity of $x\mapsto x^r$ on $[0,\infty)$. By the strong Markov property of $S$ and Lemma \ref{lem:mom}(a), $\me[(S^\leftarrow(t-n+1)-S^\leftarrow(t-n))^r]\leq \me [(S^\leftarrow(1))^r]<\infty$. Passing now to expectations in \eqref{eq:inter78} we conclude that $$\me[(A_2(t))^r]\leq \Big(\sum_{j\leq 0}\varphi(j)\Big)^r \me [(S^\leftarrow(1))^r]<\infty.$$ Combining fragments together we arrive at $\me [(\mathcal{K}(\eee^t)-A(t))^{2r}]=O((t\varphi(t))^r)$ as $t\to\infty$ and more importantly $$\me [(\mathcal{K}(\eee^t)-A_1(t))^{2r}]=O((t\varphi(t))^r),\quad t\to\infty.$$ Pick now $r\geq 1$ satisfying $r\beta>1$. Then, by Markov's inequality and the direct part of the Borel-Cantelli lemma, $$\lim_{n\to\infty}\frac{\mathcal{K}(\eee^n)-A_1(n)}{n^{1/2}\varphi(n)}=0\quad\text{a.s.}$$ In view of $$\lim_{n\to\infty}\frac{\mathcal{K}(\eee^{n+1})-A_1(n+1)}{n^{1/2}\varphi(n)}=0\quad\text{a.s.}$$ and a.s.\ monotonicity of both $t\mapsto \mathcal{K}(\eee^t)$ and $A_1$, it is enough to show that $$\lim_{n\to\infty}\frac{A_1(n+1)-A_1(n)}{n^{1/2}\varphi(n)}=0\quad\text{a.s.}$$ To prove this, write
\begin{multline*}
A_1(n+1)-A_1(n)=\int_{[0,\,n]}(\varphi(n+1-x)-\varphi(n-x)){\rm d}S^\leftarrow(x)\\+\int_{(n,\,n+1]}\varphi(n+1-x){\rm d}S^\leftarrow(x):=I_n+J_n.
\end{multline*}
By monotonicity of $\varphi$, $J_n\leq \varphi(1)(S^\leftarrow(n+1)-S^\leftarrow(n))$ a.s. Hence, by Lemma \ref{lem:mom}(c), $$\lim_{n\to\infty}\frac{J_n}{n^{1/2}\varphi(n)}=0\quad\text{a.s.}$$ Since $\Phi^\prime$ is a nonincreasing function, so is $x\mapsto \eee^{-x}\varphi^\prime(x)$. This in combination with the mean value theorem for differentiable functions enables us to conclude that, for each $x\in [0,n]$ and some $\theta_{n,x}\in [n-x, n+1-x]$, $$\varphi(n+1-x)-\varphi(n-x)=(\eee^{-\theta_{n,x}}\varphi^\prime(\theta_{n,x}))\eee^{\theta_{n,x}}\leq (\eee^{-(n-x)}\varphi^\prime(n-x))\eee^{n+1-x}=\eee \varphi^\prime(n-x).$$ As a consequence, $I_n\leq \eee \int_{[0,\,n]}\varphi^\prime(n-x){\rm d}S^\leftarrow(x)$. Assume that $\beta\geq 1$. Then $$I_n\leq \eee S^\leftarrow(n)\sup_{y\in [0,\,n]}\varphi^\prime(y)~\sim~ \eee \mu^{-1}n \varphi^\prime(n),\quad n\to\infty\quad\text{a.s.}$$ Here, the last relation follows from the strong law of large numbers for $S^\leftarrow$ and Theorem 1.5.3 in \cite{BGT}. Since
\begin{equation}\label{eq:inter23}
\lim_{n\to\infty}\frac{n\varphi^\prime(n)}{\varphi(n)}=\beta,
\end{equation}
we infer
\begin{equation}\label{eq:inter24}
\lim_{n\to\infty}\frac{I_n}{n^{1/2}\varphi(n)}=0\quad\text{a.s.}
\end{equation}
Assume now that $\beta\in (0,1]$. By Potter's bound (Theorem 1.5.6 (ii) in \cite{BGT}), given positive $B$ and $\delta\in (0,\beta)$, there exists $n_0\in\mn$ such that $$\frac{\varphi^\prime(n(1-x)}{\varphi^\prime(n)}\leq B(1-x)^{\beta-1-\delta}$$ whenever $n\geq n_0+1$ and $x\in [0, 1-n_0/n]$. Write
\begin{multline*}
\int_{[0,\,n]}\varphi^\prime(n-x){\rm d}S^\leftarrow(x)=\int_{[0,\,1-n/n_0]}\varphi^\prime(n(1-x)){\rm d}S^\leftarrow(nx)+\int_{(n-n_0,\,n]}\varphi^\prime(n-x){\rm d}S^\leftarrow(x)\\=:I_{n,1}+I_{n,2}.
\end{multline*}
In view of $I_{n,2}\leq (S^\leftarrow(n)-S^\leftarrow(n-n_0))\sup_{y\in [0,\,n_0]}\varphi^\prime(y)$, an application of Lemma \ref{lem:mom}(c) yields $$\lim_{n\to\infty}\frac{I_{n,2}}{n^{1/2}\varphi(n)}=0\quad\text{a.s.}$$ Put $\rho:=-(\beta-1-\delta)$ and note that $\rho\in (0,1)$. We proceed by estimating $I_{n,1}$: $$\frac{I_{n,1}}{n\varphi^\prime(n)}\leq B\int_{[0,\,1-n_0/n]}(1-x)^{-\rho}{\rm d}_x\Big(\frac{S^\leftarrow(nx)}{n}-\frac{x}{\mu}\Big)+\frac{B}{\mu}\int_0^1(1-x)^{-\rho}{\rm d}x.$$ The last summand is equal to $B(\mu(1-\rho))^{-1}$. Integration by parts demonstrates that the absolute value of the integral in the first summand does not exceed $$\Big(\frac{n}{n_0}\Big)^\rho \frac{|S^\leftarrow(n-n_0)-\mu^{-1}(n-n_0)|}{n}+\rho\int_0^{1-n_0/n}\frac{|S^\leftarrow(nx)-\mu^{-1}nx|}{n}(1-x)^{-\rho-1}{\rm d}x.$$ By the LIL for $S^\leftarrow$ given in formula \eqref{eq:lil}, the first summand is $O(n^{\rho-1/2}(\log\log n)^{1/2})=o(n^{1/2})$ as $n\to\infty$ a.s. The second summand is bounded from above by
\begin{multline*}
\frac{\sup_{y\in [0,\,n]}|S^\leftarrow(y)-\mu^{-1}y|}{n}\rho\int_0^{1-n_0/n}(1-x)^{-\rho-1}{\rm d}x\leq \frac{\sup_{y\in [0,\,n]}|S^\leftarrow(y)-\mu^{-1}y|}{n}\Big(\frac{n}{n_0}\Big)^\rho\\=O(n^{\rho-1/2}(\log\log n)^{1/2})=o(n^{1/2}),\quad n\to\infty\quad\text{a.s.}
\end{multline*}
Recalling \eqref{eq:inter23} we infer $$\lim_{n\to\infty}\frac{I_{n,1}}{n^{1/2}\varphi(n)}=0\quad\text{a.s.}$$ and thereupon \eqref{eq:inter24}. The proof of \eqref{eq:appr1} is complete.

At the second step we provide an appropriate a.s.\ approximation of $A_1(t)-\int_0^t \varphi(x){\rm d}x$ by the Lebesgue convolution of a Brownian motion and $\varphi^\prime$. This in combination with Proposition \ref{prop:lil} and the conclusion of Step 1 enables us to prove LILs for $A_1(t)-\mu^{-1}\int_0^t \varphi(x){\rm d}x$ and $\mathcal{K}(\eee^t)-\mu^{-1}\int_0^t \varphi(x){\rm d}x$.

\noindent {\sc Step 2}. Integrating by parts we obtain
$$A_1(t)-\mu^{-1}\int_0^t\varphi(y){\rm d}y=\varphi(0)(S^\leftarrow(t)-\mu^{-1}t)+\int_0^t (S^\leftarrow(t-x)-\mu^{-1}(t-x))\varphi^\prime(x){\rm d}x.$$
Let $W$ be a standard Brownian motion as given in Lemma \ref{chs}. Write
\begin{multline*}
\int_0^t (S^\leftarrow(t-x)-\mu^{-1}(t-x))\varphi^\prime(x){\rm d}x=\int_0^t (S^\leftarrow(t-x)-\mu^{-1}(t-x)-\sigma\mu^{-3/2}W(t-x))\varphi^\prime(x){\rm d}x\\+\sigma\mu^{-3/2}\int_0^t W(t-x)\varphi^\prime(x){\rm
d}x\bigg)=: D_1(t)+\sigma\mu^{-3/2}D_2(t).
\end{multline*}
According to formulas \eqref{eq:appr} and \eqref{eq:lil} in Lemma \ref{chs}, respectively,
\begin{multline*}
|D_1(t)|\leq\sup_{0\leq u\leq t}|S^\leftarrow(u)-\mu^{-1}u-\sigma\mu^{-3/2}W(u)|\varphi(t)\\=o\big((t\log\log t)^{1/2}\varphi(t)\big),\quad t\to\infty\quad\text{a.s.}
\end{multline*}
and $$|S^\leftarrow(t)-\mu^{-1}t|=o((t\log\log t)^{1/2}\varphi(t)),\quad t\to\infty\quad\text{a.s.}$$ By Lemma \ref{lem:haan}, the function $\varphi^\prime$ is regularly varying at $\infty$ of index $\beta-1$. With this at hand, an application of Proposition \ref{prop:lil} yields
$$C\bigg(\bigg(\frac{\int_0^t W(t-x)\varphi^\prime(x){\rm d}x}{(2(2\beta+1)^{-1}t\log\log t)^{1/2}\varphi(t)}: t>\eee\bigg)\bigg)=[-1,1]\quad\text{{\rm a.s.}}$$ In view of Step 1, replacing $\eee^t$ with $t$ we conclude that
\begin{equation}\label{eq:inter7}
C\bigg(\bigg(\frac{\mathcal{K}(t)-\mu^{-1}\int_1^t x^{-1}\Phi(x){\rm d}x}{(2\sigma^2\mu^{-3}(2\beta+1)^{-1}\log t\log\log\log t)^{1/2}\Phi(t)}: t>\eee^\eee\bigg)\bigg)=[-1,1]\quad\text{{\rm a.s.}}
\end{equation}

At the last step we have to pass from the model with the sample $E_1,\ldots, E_{\pi(t)}$ to the original model with the sample $E_1,\ldots, E_n$.

\noindent {\sc Step 3}. The basic observations are $\mathcal{K}(S_n)=K_n$ a.s. and $K_{\pi(t)}=\mathcal{K}(t)$ a.s. Put
\begin{equation*}\label{eq:inter7}
\mathcal{D}:=C\bigg(\bigg(\frac{\mathcal{K}(S_n)-\mu^{-1}\int_1^{S_n} x^{-1}\Phi(x){\rm d}x}{(2\sigma^2\mu^{-3}(2\beta+1)^{-1}\log S_n\log\log\log S_n)^{1/2}\Phi(S_n)}: n~\text{large enough}\bigg)\bigg). %\subseteq [-1,1]\quad\text{{\rm a.s.}}
\end{equation*}
In view of \eqref{eq:inter7}, $\mathcal{D} \subseteq [-1,1]$ a.s. % As a consequence of the previous steps, $${\lim\sup}_{n\to\infty}\frac{\mathcal{K}(S_n)-\mu^{-1}\int_1^{S_n} y^{-1}\Phi(y){\rm d}y}{(\log S_n\log\log\log S_n)^{1/2}\Phi(S_n)}\leq \Big(\frac{2\sigma^2}{(2\beta+1)\mu^3}\Big)^{1/2}\quad \text{a.s.}$$
Since
\begin{equation}\label{eq:inter8}
\text{the function}~~x\mapsto (\log x\log\log\log x)^{1/2}\Phi(x)~~\text{is slowly varying}
\end{equation}
and, by the strong law of large numbers for random walks, $\lim_{n\to\infty}(S_n/n)=1$ a.s., we infer $$\lim_{n\to\infty}\frac{(\log S_n\log\log\log S_n)^{1/2}\Phi(S_n)}{(\log n\log\log\log n)^{1/2}\Phi(n)}=1\quad\text{a.s.}$$ Further,
\begin{equation}\label{eq:inter9}
\Big|\int_1^{S_n}x^{-1}\Phi(x){\rm d}x-\int_1^n x^{-1}\Phi(x){\rm d}x\Big|~\sim~n^{-1}\Phi(n)|S_n-n|~\to~0,\quad n\to\infty.
\end{equation}
Here, the asymptotic equivalence is secured by the mean value theorem for integrals, the fact that the convergence $\lim_{t\to\infty}(\Phi(tx)/\Phi(t))=1$ is locally uniform in $x$ and the already mentioned strong law of large numbers for random walks. The convergence to $0$ is guaranteed by the LIL for standard random walks. Summarizing, we have proved that
\begin{equation}\label{eq:inter10}
C\bigg(\bigg(\frac{K_n-\mu^{-1}\int_1^n x^{-1}\Phi(x){\rm d}x}{(2\sigma^2\mu^{-3}(2\beta+1)^{-1}\log n \log\log\log n)^{1/2}\Phi(n)}: n~\text{large enough}\bigg)\bigg)=\mathcal{D}.
\end{equation}
On the other hand, \begin{equation*}
C\bigg(\bigg(\frac{K_{\pi(t)}-\mu^{-1}\int_1^{\pi(t)} x^{-1}\Phi(x){\rm d}x}{(2\sigma^2\mu^{-3}(2\beta+1)^{-1}\log \pi(t)\log\log\log \pi(t))^{1/2}\Phi(\pi(t))}: t~~\text{large enough}\bigg)\bigg)\subseteq \mathcal{D}.
\end{equation*}
Using $\lim_{t\to\infty}(\pi(t)/t)=1$ a.s.\ in combination with \eqref{eq:inter8} yields $$\lim_{t\to\infty}\frac{(\log \pi(t)\log\log\log \pi(t))^{1/2}\Phi(\pi(t))}{(\log t\log\log\log t)^{1/2}\Phi(t)}=1\quad\text{a.s.}$$ Similarly to \eqref{eq:inter9}, we obtain $$\Big|\int_1^{\pi(t)}x^{-1}\Phi(x){\rm d}x-\int_1^t x^{-1}\Phi(x){\rm d}x\Big|~\sim~t^{-1}\Phi(t)|\pi(t)-t|~\to~0,\quad t\to\infty$$ having utilized $|\pi(t)-t|=O((t\log\log t)^{1/2})$ as $t\to\infty$ a.s. The latter follows from the LIL for renewal processes, see, for instance, Proposition 3.5 in \cite{Iksanov+Jedidi+Bouzzefour:2017}. Thus, $[-1,1]\subseteq \mathcal{D}$ a.s. and thereupon  $\mathcal{D}=[-1,1]$ a.s. Taking into account \eqref{eq:inter10} completes the proof.

\section{Proof of Corollary \ref{cor:appl}}\label{sect:appl}

(a) Assume that we can check that
\begin{equation}\label{eq:inter35}
\Phi(t)=\theta (\log t+\gamma-\log \lambda)+o(1),\quad t\to\infty,
\end{equation}
where $\gamma$ is the Euler-Mascheroni constant. Then condition \eqref{eq:Haan} holds with $\beta=1$ and $\ell(t)=\theta$ for $t>0$. Further,
$$\mu=\theta\int_0^\infty \eee^{-\lambda x}{\rm d}x=\theta\lambda^{-1}<\infty\quad\text{and}\quad \sigma^2=\theta\int_0^\infty x \eee^{-\lambda x}{\rm d}x=\theta\lambda^{-2}\in (0,\infty).$$ Thus, Theorem \ref{thm:main} applies. Relation \eqref{eq:inter35} ensures that $$\int_1^n \frac{\Phi(y)}{y}{\rm d}y=\frac{\theta(\log n)^2}{2}+\theta(\gamma-\log \lambda)\log n+o(\log n),\quad n\to\infty.$$ This demonstrates that $(\lambda/2)(\log n)^2$ is the correct centering.

We shall need an integral representation for $\Psi$ the logarithmic derivative of the gamma function
\begin{equation}\label{eq:logder}
\Psi(s)=-\gamma+\int_0^1 y^{-1}(1-(1-y)^{s-1}){\rm d}y,\quad s>0.
\end{equation}
Now we prove \eqref{eq:inter35}. Changing the variable $y=1-\eee^{-x}$ we obtain
\begin{multline*}
\theta^{-1}\Phi(t)=\int_0^\infty (1-\exp(-t(1-\eee^{-x})))\frac{\eee^{-\lambda x}}{x}{\rm d}x=\int_0^1 (1-\exp(-ty))\frac{(1-y)^{\lambda-1}}{y}{\rm d}y\\+\int_0^1 (1-y)^{\lambda-1}\Big(\frac{1}{|\log(1-y)|}-\frac{1}{y}\Big){\rm d}y-\int_0^1 \exp(-ty)(1-y)^{\lambda-1}\Big(\frac{1}{|\log(1-y)|}-\frac{1}{y}\Big){\rm d}y\\:=A(t)+B-C(t).
\end{multline*}
Plainly, $\lim_{t\to\infty} C(t)=0$ by the monotone convergence theorem. Further, $$B=\int_0^1\frac{1-(1-y)^{\lambda-1}}{y}{\rm d}y-\int_0^1 \frac{1-(1-y)^{\lambda-1}}{|\log (1-y)|}{\rm d}y=\Psi(\lambda)+\gamma-\log \lambda.$$ By a change of variable the last integral transforms into a Frullani integral which is equal to $\log \lambda$. Finally, $$A(t)=\int_0^t \frac{1-\eee^{-y}}{y}{\rm d}y-\int_0^1\frac{1-(1-y)^{\lambda-1}}{y}{\rm d}y\\+\int_0^1\eee^{-ty} \frac{1-(1-y)^{\lambda-1}}{y}{\rm d}y.$$ The penultimate summand is equal to $-\gamma-\Psi(\lambda)$ and the last summand is $o(1)$ as $t\to\infty$ by the monotone convergence theorem. The first summand is equal to $$\log t+\int_0^1 \frac{1-\eee^{-y}}{y}{\rm d}y -\int_1^\infty \frac{\eee^{-y}}{y}{\rm d}y+\int_t^\infty \frac{\eee^{-y}}{y} {\rm d}y=\log t+\gamma+o(1),\quad t\to\infty.$$ We have shown that
\begin{equation}\label{eq:inter36}
A(t)=\int_0^1 (1-\exp(-ty))\frac{(1-y)^{\lambda-1}}{y}{\rm d}y=\log t-\Psi(\lambda)+o(1),\quad t\to\infty.
\end{equation}

Combining fragments together we obtain \eqref{eq:inter35}.

\noindent (b) According to \eqref{eq:inter36}
\begin{equation}\label{eq:inter34}
\Phi(t)=\theta \log t-\theta\Psi(\lambda)+o(1),\quad t\to\infty.
\end{equation}
This entails \eqref{eq:Haan} with $\beta=1$ and $\ell(t)=\theta$ for $t>0$. Differentiating \eqref{eq:logder} % an integral representation $\Psi(s)=-\gamma+\int_0^1 y^{-1}(1-(1-y)^{s-1}){\rm d}y$, where $\gamma$ is the Euler-Mascheroni constant,
we infer $$\mu=\theta\int_0^\infty x \frac{\eee^{-\lambda x}}{1-\eee^{-x}}{\rm d}x=\theta\int_0^1 |\log(1-y)|\frac{(1-y)^{\lambda-1}}{y}{\rm d}y=\theta\Psi^\prime(\lambda)<\infty$$ and $$\sigma^2=\theta\int_0^1 (\log(1-y))^2\frac{(1-y)^{\lambda-1}}{y}{\rm d}y=-\theta\Psi^{\prime\prime}(\lambda)\in (0,\infty).$$ Hence, Theorem \ref{thm:main} applies. In view of \eqref{eq:inter34}, $$\int_1^n \frac{\Phi(y)}{y}{\rm d}y=\frac{\theta(\log n)^2}{2}-\theta\Psi(\lambda)\log n+o(\log n)=\frac{\theta(\log n)^2}{2}+o((\log n)^{3/2}),\quad n\to\infty.$$ This proves that $(2\Psi^\prime(\lambda))^{-1}(\log n)^2$ is the right centering.

%To prove \eqref{eq:inter34}, we write by changing the variable $y=1-\eee^{-x}$
%\begin{multline*}
%\theta^{-1}\Phi(t)=\int_0^\infty (1-\exp(-t(1-\eee^{-x})))\frac{\eee^{-\lambda x}}{1-\eee^{-x}}{\rm d}x=\int_0^t \frac{1-\eee^{-y}}{y}{\rm d}y-\int_0^1\frac{1-(1-y)^{\lambda-1}}{y}{\rm d}y\\+\int_0^1\eee^{-ty} \frac{1-(1-y)^{\lambda-1}}{y}{\rm d}y.
%\end{multline*}
%The penultimate summand is equal to $-\gamma-\Psi(\lambda)$ and the last summand is $o(1)$ as $t\to\infty$ by the monotone convergence theorem. The first summand is equal to $$\log t+\int_0^1 \frac{1-\eee^{-y}}{y}{\rm d}y -\int_1^\infty \frac{\eee^{-y}}{y}{\rm d}y+\int_t^\infty \frac{\eee^{-y}}{y} {\rm d}y=\log t+\gamma+o(1),\quad t\to\infty.$$ Combining fragments together we obtain \eqref{eq:inter34}.\

The proof of Corollary \ref{cor:appl} is complete.

\vskip0.5cm
\noindent
{\bf Acknowledgement}. The work of Wissem Jedidi was supported by the Research Supporting Project (RSP2024R162), King Saud University, Riyadh, Saudi Arabia.

\end{document}